\documentclass[12pt, reqno]{amsart}
 \usepackage{amsmath, amsthm, amscd, amsfonts, amssymb, graphicx, color, float}
\usepackage[bookmarksnumbered, colorlinks, plainpages]{hyperref}
\input{mathrsfs.sty}
\hypersetup{colorlinks=true,linkcolor=red, anchorcolor=green, citecolor=cyan, urlcolor=red, filecolor=magenta, pdftoolbar=true}

\newtheorem{theorem}{Theorem}[section]
\newtheorem{lemma}[theorem]{Lemma}
\newtheorem{proposition}[theorem]{Proposition}
\newtheorem{corollary}[theorem]{Corollary}
\theoremstyle{definition}
\newtheorem{definition}[theorem]{Definition}
\newtheorem{example}[theorem]{Example}

\theoremstyle{remark}
\newtheorem{remark}[theorem]{Remark}
\numberwithin{equation}{section}

\begin{document}

\title[Separated Pairs of Submodules in Hilbert $C^*$-modules]{Separated Pairs of Submodules in Hilbert $C^*$-modules}

\author[R. Eskandari,  W. Luo, M.S. Moslehian, Q. Xu, H. Zhang]{R. Eskandari$^{1}$,  W. Luo$^2$, M. S. Moslehian$^3$, Q. Xu$^4$ \MakeLowercase{and} H. Zhang $^5$ }

\address{$^1$Department of Mathematics Education, Farhangian University, P.O. Box 14665-889, Tehran, Iran.}
\email{Rasoul.eskandari@cfu.ac.ir; eskandarirasoul@yahoo.com}

\address{$^2$Department of Statistics and Mathematics, Shanghai Lixin University of Accounting and
	Finance, Shanghai 201209, PR China}
\email{luoweipig1@163.com}

\address{$^3$Department of Pure Mathematics, Center of Excellence in Analysis on Algebraic Structures (CEAAS), Ferdowsi University of Mashhad, P. O. Box 1159, Mashhad 91775, Iran.}
\email{moslehian@um.ac.ir; moslehian@yahoo.com}

\address{$^4$Department of Mathematics, Shanghai Normal University, Shanghai 200234, PR China}
\email{qingxiang\_xu@126.com}

\address{$^4$School of Mathematics and Statistics, Shangqiu Normal University, Shangqiu 476000, PR China.}
\email{csqam@163.com}

\renewcommand{\subjclassname}{\textup{2020} Mathematics Subject Classification} 


\begin{abstract}
We introduce the notion of the separated pair of closed submodules in the setting of Hilbert $C^*$-modules. We demonstrate that even in the case of Hilbert spaces this concept has several nice characterizations enriching the theory of separated pairs of subspaces in Hilbert spaces. Let $\mathscr H$ and $\mathscr K$ be orthogonally complemented closed submodules of a Hilbert $C^*$-module $\mathscr E$. We establish that $ (\mathscr H,\mathscr K)$ is a separated pair in $\mathscr{E}$ if and only if there are idempotents $\Pi_1$ and $\Pi_2$ such that $\Pi_1\Pi_2=\Pi_2\Pi_1=0$ and $\mathscr R(\Pi_1)=\mathscr H$ and $\mathscr R(\Pi_2)=\mathscr K$. We show that $\mathscr R(\Pi_1+\lambda\Pi_2)$ is closed for each $\lambda\in \mathbb{C}$ if and only if $\mathscr R(\Pi_1+\Pi_2)$ is closed.
	
We use the localization of Hilbert $C^*$-modules to define the angle between closed submodules. We prove that if $(\mathscr H^\perp,\mathscr K^\perp)$ is concordant, then  $(\mathscr H^{\perp\perp},\mathscr K^{\perp\perp})$ is a separated pair if the cosine of this angle is less than one. We also present some surprising examples to illustrate our results.
\end{abstract}

\maketitle
%====================================================%
%====================================================%
%====================================================%
\section{Introduction}\label{sec:introduction}
Throughout this paper, $\mathscr{A}$ denotes a $C^*$-algebra. We denote by  $\mathrm{S}(\mathscr A ) $ and  $\mathrm{PS}(\mathscr A ) $  the set of states and the set of pure states on  $\mathscr{A}$, respectively. We assume that  $\mathscr{E}$ and $\mathscr F$ are Hilbert $C^*$-modules over $\mathscr A$.  The set of all adjointable operators from ${\mathscr E}$ into $\mathscr{F}$ is represented by $\mathcal{L}(\mathscr{E}, \mathscr{F})$, with the abbreviation $\mathcal{L}(\mathscr{E})$ if $\mathscr{E}=\mathscr{F}$. When we deal with a Hilbert space $\mathcal{H}$, we denote  $\mathcal{L}(\mathcal{H})$ by  $\mathbb{B}(\mathcal{H})$. The identity of an algebra  is denoted  by $I$. We signify by $\mathscr R(T)$ and $\mathscr N(T)$ the range and nullity of an operator $T$, respectively. For more information about Hilbert $C^*$-modules and their geometry, see \cite{FRA1, Lan, MT}.

A submodule $\mathscr M\subseteq \mathscr E$ is called \emph{orthogonally complemented} in $\mathscr E$ if $\mathscr M \oplus\mathscr M^\perp=\mathscr E$, where $\mathscr M^\perp=\{x\in \mathscr E:\langle x,y\rangle=0,\mbox{~for all~} y\in \mathscr M\}$. In this case $\mathscr M$ is closed, and we use the notation $P_ \mathscr M$ to denote the projection from $\mathscr{E}$ onto $\mathscr M$. Unlike Hilbert spaces, a closed submodule is not necessarily orthogonally complemented. If $T\in\mathcal{L}(\mathscr E,\mathscr F)$ has closed range, then $\mathscr R(T)$ and $\mathscr N(T)$ are orthogonally complemented; see \cite[Theorem 3.2]{Lan}.

In the framework of Hilbert spaces, given a bounded linear idempotent $\Pi$, let $P$ and $Q$ be the projections onto the range $\mathscr{R}(\Pi)$ and the null space $\mathscr{N}(\Pi)$ of $\Pi$, respectively. One research field is concerned with the relationships between $\Pi$, $P$, and $Q$, and some interesting results can be found in the literature; see \cite{Ando, Bottcher-Spitkovsky}.

A subspace $\mathscr M$ of Hilbert space $\mathscr E$ is said to be an \emph{operator range} if there is a bounded linear operator $A$ such that $\mathscr M=\mathscr{R}(A)$. The set $\mathcal{L}$ of all operator ranges constitutes a lattice with respect to the vector addition and the set intersection. It is known that an operator range
$\mathscr R(A)$ is complemented in the lattice $\mathcal{L}$ (in the sense that there is an operator range $\mathscr R(B)$ such that $\mathscr R(A)\cap\mathscr R(B)=0$ and $\mathscr R(A)+\mathscr R(B)$ is closed) if and only if $\mathscr R(A)$ is closed; see \cite[Theorem 2.3]{Fillmore}.
In \cite[Proposition 3.7]{Dixmier}, the nonclosedness of the sum of two disjoint operator ranges is studied.

It is clear that for any idempotent $\Pi\in \mathscr{L}(\mathscr{E})$,
\begin{equation}\label{equ:observation-1}
\mathscr {R}(\Pi)\cap\mathscr {R}(I-\Pi)=0\quad\mbox{and}\quad \mathscr{R}(\Pi)+\mathscr{R}(I-\Pi)=\mathscr{E}.
\end{equation}
Motivated by this, we give the following key concept.

\begin{definition}
	Let $\mathscr H$ and $\mathscr K$ be closed submodules of $\mathscr E$. Then we say that $(\mathscr H,\mathscr K)$ is a \emph{separated pair} if
	\begin{equation}\label{eq2a}
	\mathscr H\cap\mathscr K=0 \mathrm{~and }~ \mathscr H+\mathscr K~\mathrm{is ~orthogonally~complemented~in~} \mathscr E.
	\end{equation}
\end{definition}
%%%%%%%%%%%%%%%%%%%%%%%%%%%%%%%%%%%%%%%%%%%%%%%%%%%%%%%%%%%%%%%%%%%%%%
\begin{lemma}\label{lem:ranges of P plus Q equals-1}{\rm (see\,\cite[Proposition~4.6]{Luo-Song-Xu}, \cite[Remark~5.8]{Luo-Moslehian-Xu}, and \cite[Theorem~1]{Tan-Xu-Yan})}
	Let $P,Q\in\mathcal{L}(\mathscr{E})$ be projections. Then the following statements are all equivalent:
	\begin{enumerate}
		\item[{\rm (i)}] $\mathscr{R}(P+Q)$ is closed in $\mathscr{E}$;
		\item[{\rm (ii)}] $\mathscr{R}(P)+\mathscr{R}(Q)$ is closed in $\mathscr{E}$;
		\item[{\rm (iii)}] $\mathscr{R}(I-P)+\mathscr{R}(I-Q)$ is closed in $\mathscr{E}$;
		\item[{\rm (iv)}] $\mathscr{R}(2I-P-Q)$ is closed in $\mathscr{E}$;
		\item[{\rm (v)}] For every complex numbers $\lambda_1$ and $\lambda_2$, $\mathscr{R}(\lambda_1 P+\lambda_2 Q)$ is closed in $\mathscr{E}$.
	\end{enumerate}
	In each case, we have
	\begin{align*}&\mathscr{R}(P)+\mathscr{R}(Q)=\mathscr{R}(P+Q),\\
	&\mathscr{R}(I-P)+\mathscr{R}(I-Q)=\mathscr{R}(2I-P-Q).
	\end{align*}
\end{lemma}

%%%%%%%%%%%%%%%%%%%%%%%%%%%%%%%%%%%%%%%%%%%%%%%%%%%%%%%%%%%%%%%%%%%%%%%%%%%
\begin{remark}\label{rm separted}
	Let $\mathscr H$ and $\mathscr K$ be orthogonally complemented closed submodules and let $P$ and $Q$ be projections onto $\mathscr H$ and $\mathscr K$, respectively. It follows  from \cite[Lemma~2.3]{Luo-Moslehian-Xu} that $\overline{\mathscr H+\mathscr K}=\overline{\mathscr R(P+Q)}$. Therefore, \cite[Theorem 3.2]{Lan}  and Lemma~\ref{lem:ranges of P plus Q equals-1} entail that for orthogonally complemented closed submodules \eqref{eq2a} is equivalent to
	\[
	\mathscr H\cap\mathscr K=0 \mathrm{~and }~ \mathscr H+\mathscr K~\mathrm{is ~closed~in~} \mathscr E.
	\].
\end{remark}	
Note that for every idempotent $\Pi$ on $\mathcal{E}$, the pair $\big(\mathscr R(\Pi), \mathscr {R}(I-\Pi)\big)$ is always separated.

Our investigation of the separated pairs is also motivated by the Dixmier angle. Let $\mathscr H$ and $\mathscr K$ be two orthogonally complemented closed submodules of a Hilbert $C^*$-module $\mathscr{E}$. The \emph{Dixmier angle (minimum angle)}, denoted by $\alpha_{0} (\mathscr H,\mathscr K)$, is the unique angle in $[0,\frac{\pi}{2}]$ whose cosine is equal to $c_{0}(\mathscr H,\mathscr K)$, where
\begin{equation}\label{definition of Dixmier angle}
c_{0}(\mathscr H,\mathscr K)=\sup\big\{ \| \langle x,y\rangle \|: x\in \mathscr H, y\in \mathscr K, \Vert x\Vert\le 1, \Vert y\Vert\le 1\big\}.
\end{equation}
As in the Hilbert space case  (see \cite[Lemma~2.10]{Deutsch}), it can be shown  that
\begin{equation}\label{equ:formula for c H K++}c_{0}(\mathscr H,\mathscr K)=\Vert P_\mathscr H P_\mathscr K \Vert.
\end{equation}
Hence, $\|P_\mathscr HP_\mathscr K\|<1$ if and only if $\alpha_{0} (\mathscr H,\mathscr K)>0$ and this occurs if and only if $(\mathscr H,\mathscr K)$ is a separated pair as the following characterization of such separated pairs of submodules shows.

\begin{lemma}\label{lem:norm of two projections less than one}{\rm \cite[Lemma~5.10]{Luo-Moslehian-Xu}}  Let $\mathscr{H}$ be a Hilbert module over the $C^*$-algebra $\mathfrak{A}$, and let  $P,Q\in\mathcal{L}(\mathscr H)$ be projections. Then the following statements are equivalent:
	\begin{enumerate}
		\item[{\rm (i)}] $\Vert PQ\Vert<1$;
		\item[{\rm (ii)}] $\mathscr{R}(P)\cap \mathscr{R}(Q)=0$ and $\mathscr{R}(P)+\mathscr{R}(Q)$ is closed;
		\item[{\rm (iii)}] $\mathscr{R}(I-P)+\mathscr{R}(I-Q)=\mathscr H$.
	\end{enumerate}
\end{lemma}

The paper is organized as follows. In Section~\ref{sec:characterizations of complemented projections}, we focus on the study of a separated pair of orthogonally complemented closed submodules in terms of the associated idempotents. Our first result is Theorem~\ref{th equivalent of separation},
where two natural idempotents are constructed when a separated pair is given. The key point of this construction is the verification of the adjointability of these two idempotents.  As an application, the Moore-Penrose invertibility of certain operators associated with a separated pair is clarified. Specifically, a generalization of \cite[Theorem~3.8]{Ando} is obtained; see Proposition~\ref{thm:sth concerns M-P inverse} for the details. We also present an example showing that there is a separated pair $(\mathscr H,\mathscr K)$ and idempotents $\Pi_1,\Pi_2$ with $\mathscr R(\Pi_1)=\mathscr H$ and $\mathscr R(\Pi_2)=\mathscr K$ such that $\mathscr R(\Pi_1+\Pi_2)$ is not closed. An alternative description of the idempotents constructed in the proof of Theorem~\ref{th equivalent of separation} is given.

In Section ~\ref{section HC}, we   generalize the Dixmier and Friedichs angle  between two closed submodules. Let $\mathscr E$ be a Hilbert $C^*$-module over a $C^*$-algebra $\mathscr A$, and let $\mathscr H$ and $\mathscr K$ be closed submodules of $\mathscr E$.  Denote by $\mathrm{S}(\mathscr A ) $   the set of all states on $\mathscr{A}$. Following the constructions presented in \cite{Pa}, we associate  to each $f\in \mathrm{S}(\mathscr A )$ a Hilbert space $\mathscr H_f$. The term of the concordant pair was introduced recently  in \cite[Definition~1.1]{Bram} under the restriction that $\mathscr H$ and $\mathscr K$ are both orthogonally complemented in $\mathscr E$. In our Definition~\ref{def:concordant}, such a restriction of orthogonal complementarity is no longer employed. A necessary and sufficient condition is provided in Theorem~\ref{th concorant} under which the pair $(\mathscr H,\mathscr K)$ is concordant.
At the end of this section, the cosine of the Dixmier and Friedrichs angle between $\mathscr H$ and $\mathscr K$ is investigated. Finally, a sufficient condition of the separated pair is provided in the last theorem of this section; see Theorem~\ref{thm:local char of separated pair}.

\section{separated pair of orthogonally complemented submodules and Idempotents}\label{sec:characterizations of complemented projections}

We start this section with the following theorem in which we give some equivalent conditions to ensure that a pair of orthogonally complemented closed submodules is separated.
\begin{theorem}\label{th equivalent of separation}
	Let $\mathscr H$ and $\mathscr K$ be  orthogonally complemented closed submodules of $\mathscr E$. The following statements are equivalent:
	\begin{itemize}
		\item[{\rm (i)}]  $(\mathscr H,\mathscr K)$ is a separated pair.
		\item[{\rm (ii)}] There are idempotents $\Pi_1$ and $\Pi_2$ in $\mathcal{L}(\mathscr E)$ such that $\Pi_1\Pi_2=\Pi_2\Pi_1=0$, $\mathscr R(\Pi_1)=\mathscr H$ and $\mathscr R(\Pi_2)=\mathscr K$.
		\item[{\rm (iii)}] There is an idempotent $\Pi\in \mathcal{L}(\mathscr E)$ such that $\mathscr R(\Pi)=\mathscr H$ and $\mathscr K\subseteq\mathscr N(\Pi)$.
	\end{itemize}
\end{theorem}
\begin{proof}
	(i)$\Longrightarrow $(ii). Suppose that $(\mathscr H,\mathscr K)$ is separated. Define $\Pi_i:\mathscr H+\mathscr K\to \mathscr E (i=1,2)$ by
	\begin{equation}\label{defn of Pi 1 2}
	\Pi_1(x+y)=x, \quad \Pi_2(x+y)=y\qquad (x\in\mathscr H, y\in \mathscr K)\,.
	\end{equation}
	Since $\mathscr H\cap \mathscr K=0$, it is observed that $\Pi_i (i=1,2)$ are well-defined. A simple use of the closed graph theorem shows that  $\Pi_1$ and $\Pi_2$ are both bounded.

	Let $\Pi_1$ be extended to a bounded linear operator on $\mathscr E$ (denoted still by $\Pi_1$) by putting $\Pi_1 z=0$ for all $z\in \big(\mathscr H+\mathscr K\big)^\perp$. Then
	\[
	\Pi_1^2\big(x+y+z\big)=\Pi_1\big(x\big)=x=\Pi_1(x+y+z)\,,
	\]
	where $x\in \mathscr H,y\in \mathscr K$ and $z\in \big(\mathscr H+\mathscr K\big)^\perp$. This shows that $\Pi_1^2=\Pi_1$, $\mathscr R(\Pi_1)=\mathscr H$ and $\mathscr K\subseteq\mathscr N(\Pi_1)$.
	
	Next, we show that $\Pi_1\in \mathcal{L}(\mathscr E)$. For this, we prove that
	$\mathscr E$ is the direct sum of $\mathscr H^\perp$ and $\mathscr K^\perp\cap (\mathscr H+\mathscr K)$
	
	Indeed, since $(\mathscr H,\mathscr K)$ is separated, by Lemma~\ref{lem:norm of two projections less than one} we have $\mathscr E=\mathscr H^\perp+\mathscr K^\perp$.
	So, for every $x\in \mathscr E$, there exist $x_1\in \mathscr H^\perp$ and $x_2\in \mathscr K^\perp$ such that $x=x_1+x_2$. Let $x_2=z+z'$, where $z\in \mathscr H+\mathscr K$ and $z'\in (\mathscr H+\mathscr K)^\perp $. Then $z=x_2-z'\in \mathscr K^\perp+(\mathscr H+\mathscr K)^\perp=\mathscr K^\perp$, hence $z\in \mathscr K^\perp\cap(\mathscr H+\mathscr K)$. It follows that
	\[
	x=x_1+x_2=(x_1+z')+z\in \mathscr H^\perp+\left[\mathscr K^\perp\cap(\mathscr K+\mathscr H\right].
	\]
	Hence,
	$$\mathscr E=\mathscr H^\perp+\big[\mathscr K^\perp\cap (\mathscr H+\mathscr K)\big].$$
	Furthermore, it is easy to verify
	$$\mathscr H^\perp\cap \big[\mathscr K^\perp\cap (\mathscr H+\mathscr K)\big]=0.$$
	
	Now, we define unambiguously an operator $D:\mathscr E\to \mathscr E$ by
	\[
	D(x+y)=y
	\]
	for all $x\in \mathscr H^\perp$ and $y\in\mathscr K^\perp\cap (\mathscr K+\mathscr H)$.  For every such $x$ and $y$, together with $h\in \mathscr H$,
	$k\in \mathscr K$ and $z\in(\mathscr H+\mathscr K)^\perp$, we have
	\begin{align*}
	\big\langle \Pi_1(h+k+z),x+y\big\rangle=&\langle h,x+y\rangle=\langle h,y\rangle=\langle h+k+z,y\rangle\\
	=&\big\langle h+k+z,D(x+y)\big\rangle.
	\end{align*}
	Therefore, $\Pi_1\in\mathcal{L}(\mathscr E)$ and $\Pi_1^*=D$.
	
	The operator $\Pi_2$, defined by \eqref{defn of Pi 1 2}, can be extended to an idempotent in $\mathcal{L}(\mathscr E)$ such that  $\Pi_1\Pi_2=\Pi_2\Pi_1=0$,
	$\mathscr R(\Pi_1)=\mathscr H$ and $\mathscr R(\Pi_2)=\mathscr K$.
	
	(ii)$\Longrightarrow $(iii) It can be derived immediately if we put $\Pi=\Pi_1$.

	(iii)$\Longrightarrow $(i). Let $\Pi\in \mathcal{L}(\mathscr E)$ be an idempotent satisfying $\mathscr R(\Pi)=\mathscr H$ and $\mathscr K\subseteq\mathscr N(\Pi)$.
	Then $\mathscr H\cap\mathscr K=0$, since
	\[
	\mathscr H\cap\mathscr K\subseteq \mathscr{R}(\Pi)\cap\mathscr{N}(\Pi)=0.\]
	Now let $\{x_n\}$ and $\{y_n\}$ be sequences in $\mathscr E$ such that
	\begin{equation}\label{limit1}
	\lim_{n\to\infty}(\Pi x_n+P_\mathscr K y_n)= z\,.
	\end{equation}
	Then
	\begin{equation}\label{limit12}
	\Pi x_n=\Pi(\Pi x_n+P_\mathscr K y_n)\to \Pi z.
	\end{equation}
	Employing \eqref{limit1} and \eqref{limit12}, we see that there is $y\in \mathscr E$ such that $\lim\limits_{n\to \infty}P_\mathscr K y_n=P_\mathscr K y$. Hence, $z=\Pi z+P_\mathscr K y$. This shows that $\mathscr H+\mathscr K$ is closed. Thus $(\mathscr H,\mathscr K)$ is a separated pair.
\end{proof}
%==================================================================================

\begin{corollary}\label{corollary abs} Let $\mathscr H$ and $\mathscr K$ be  orthogonally complemented closed submodules of $\mathscr E$. The following statements are equivalent:
	\begin{itemize}
		\item[{\rm (i)}] $(\mathscr H,\mathscr K)$ is a separated pair.
		\item[{\rm (ii)}] There exist  constants $\alpha_1,\alpha_2 >0$ such that
		\begin{align*}
		|x+y|\geq \alpha_1 |x|\mbox{~and~}	|x+y|\geq \alpha_2 |y| \quad (x\in \mathscr H,y\in \mathscr K).
		\end{align*}
		\item[{\rm (iii)}] There exist  constants  $\alpha_1,\alpha_2 >0$ such that
		\begin{equation}\label{separeted-norm absolute in HC}
		\|x+y\|\geq \alpha_1 \|x\|\mbox{~and~}	\|x+y\|\geq \alpha_2 \|y\| \quad (x\in \mathscr H,y\in \mathscr K).
		\end{equation}
	\end{itemize}
\end{corollary}
\begin{proof} (i) $\Longrightarrow$(ii). For $i=1,2$, let $\Pi_i$ be as in the proof of the preceding theorem, and let  $\alpha_i=\|\Pi_i\|$. From \eqref{defn of Pi 1 2}, we see that for every $x\in\mathscr H$ and $y\in \mathscr K$,
	\begin{align*}\langle x,x\rangle=\big\langle \Pi_1 (x+y),\Pi_1(x+y)\big\rangle=\big\langle \Pi_1^*\Pi_1 (x+y),x+y\big\rangle\le \alpha_1^2\,\langle x+y,x+y\rangle.
	\end{align*}
	Hence,
	$$|x|=\langle x,x\rangle^\frac12\le \alpha_1 \langle x+y,x+y\rangle^\frac12=\alpha_1 |x+y|.$$
	Similarly, $|y|\le \alpha_2 |x+y|$.
	
	(ii)$\Longrightarrow$(iii). It follows from
	\begin{align*}a,b\in \mathscr A,a\geq b>0\Rightarrow \|a\|\geq\|b\|\,.
	\end{align*}

	(iii)$\Longrightarrow$(i). Suppose that \eqref{separeted-norm absolute in HC} is valid. It is easy to see that $\mathscr H+\mathscr K$ is closed and $\mathscr H\cap \mathscr K=0$. Hence by Remark \ref{rm separted}, $(\mathscr H,\mathscr K)$ is separated.
\end{proof}
%==================================================================================

Given two arbitrary idempotents $\Pi_1$ and $\Pi_2$ on a Hilbert space, it is shown in \cite{Du} that the invertibility of the linear combination $\lambda_1\Pi_1+\lambda_2 \Pi_2$ is
independent of the choice of $\lambda_i, i=1,2$, if $\lambda_1\lambda_2\neq 0$ and $\lambda_1+\lambda_2\neq 0$. Such  a result can be generalized as follows.

\begin{theorem}\label{lem closed range of idempotents}
	Let $(\mathscr H,\mathscr K)$ be a  separated pair of orthogonally complemented submodules of $\mathscr E$. Let $\Pi_1$ and $\Pi_2$ be idempotents  in $\mathcal L(\mathscr E)$ 
	such that $\mathscr R(\Pi_1)=\mathscr H$ and $\mathscr R(\Pi_2)=\mathscr K$. Then the following assertions are equivalent:
	\begin{itemize}
		\item[{\rm (i)}] $\mathscr R(\Pi_1+\lambda\Pi_2)$ is closed in $\mathscr{E}$ for every $\lambda\in \mathbb{C}$;
		\item[{\rm(ii)}] $\mathscr R(\Pi_1+\lambda\Pi_2)$ is closed in $\mathscr{E}$ for some $\lambda\in \mathbb{C}\setminus \{0\}$;
		\item[{\rm (iii)}] $\mathscr R(\Pi_1+\Pi_2)$ is closed in $\mathscr{E}$.
	\end{itemize}
\end{theorem}
\begin{proof}
	
	(i)$\Longrightarrow$(ii). It is clear.\\
	(ii)$\Longrightarrow$(iii). Suppose that $\mathscr R(\Pi_1+\lambda\Pi_2)$ is closed for some $\lambda\in \mathbb{C}\setminus\{0\}$. Let $x\in \overline{\mathscr R(\Pi_1+\Pi_2)}\subseteq \overline{\mathscr R(\Pi_1)+\mathscr R(\Pi_2)}=\mathscr R(\Pi_1)+\mathscr R(\Pi_2)$. Then there exist a sequence $\{x_n\}$ in $\mathscr E$ and $z_1,z_2\in \mathscr E$ such that
	\begin{equation}\label{limi1}
	x=\lim_{n\to \infty} (\Pi_1+\Pi_2)x_n=\Pi_1z_1+\Pi_2z_2\,.
	\end{equation}
	Since the pair  $(\mathscr R(\Pi_1), \mathscr R(\Pi_2))$ is separated, it follows from  Corollary \eqref{corollary abs} that
	\[
	\|\Pi_1(x_n-z_1)+\Pi_2( x_n-z_2)\|\geq \alpha\|\Pi_2( x_n-z_i)\|
	\]for some $\alpha>0$. Hence
	\begin{equation}\label{limi2}
	\lim_{n\to \infty} \Pi_1x_n=\Pi_1z_1\,,\qquad \lim_{n\to \infty}\Pi_2x_n=\Pi_2z_2\,,
	\end{equation}
	which gives
	\begin{align*}
	\lim_{n\to \infty} (\Pi_1+\lambda \Pi_2)x_n=\Pi_1 z_1+\lambda \Pi_2z_2.
	\end{align*}
	Consequently,
	\begin{equation}\label{reduce to z0-1}\Pi_1 z_1+\lambda \Pi_2z_2=(\Pi_1+\lambda \Pi_2)z_0\end{equation}
	for some $z_0\in \mathscr R(\Pi_1)$, since $\mathscr R(\Pi_1+\lambda\Pi_2)$ is assumed to be closed. Therefore,
	$$\Pi_1(z_1-z_0)=\Pi_2\big(\lambda(z_0-z_2)\big)\in \mathscr R(\Pi_1)\cap\mathscr R(\Pi_2)=0.$$
	Substituting $\Pi_1 z_1=\Pi_1 z_0$ in \eqref{reduce to z0-1} yields $\lambda \Pi_2z_2=\lambda \Pi_2z_0$, which in turn gives
	$\Pi_2z_2=\Pi_2z_0$, since $\lambda\ne 0$. It follows from \eqref{limi1} that $x=(\Pi_1+\Pi_2)z_0$.
	This shows the closedness of $\mathscr R(\Pi_1+\Pi_2)$.\\
	(iii)$\Longrightarrow$(i). The case of $\lambda=0$ is trivial.
	Suppose that $\mathscr R(\Pi_1+\Pi_2)$ is closed. Let $\lambda\in \mathbb{C}\setminus\{0\}$ be arbitrary. If $x\in \overline{\mathscr R(\Pi_1+\lambda\Pi_2)}\subseteq \mathscr R(\Pi_1)+\mathscr R(\Pi_2)$, then there exists a sequence $\{x_n\}$ in $\mathscr R(\Pi_2)$ such that
	\begin{equation}\label{lim1}
	\lim_{n\to \infty} (\Pi_1+\lambda\Pi_2)x_n=x=\Pi_1z_1+\Pi_2z_2
	\end{equation}
	for some $z_1,z_2\in \mathscr E$.
	Since $(\mathscr R(\Pi_1), \mathscr R(\Pi_2))$ is separated, by the same as reasoning in getting \eqref{limi2}, we have
	\begin{equation*}\label{lim2}
	\lim_{n\to \infty} \Pi_1x_n=\Pi_1z_1\quad {\rm and} \quad \lim_{n\to \infty}\lambda \Pi_2x_n=\Pi_2z_2\,.
	\end{equation*}
	As a result,
	$$\lim_{n\to \infty} (\Pi_1+\Pi_2)x_n=\Pi_1z_1+\Pi_2\left(\frac{z_2}{\lambda}\right)=(\Pi_1+\Pi_2)z_0$$
	for some $z_0\in \mathscr E$.
	Therefore, $\Pi_1 z_1=\Pi_1 z_0$ and $\Pi_2 z_2=\lambda\Pi_2(z_0)$, since $\mathscr{R}(\Pi_1)\cap\mathscr{R}(\Pi_2)=0$. It follows from
	\eqref{lim1} that
	$x=(\Pi_1+\lambda\Pi_2)(z_0)\in \mathscr R(\Pi_1+\lambda\Pi_2)$.  Hence, $\mathscr R(\Pi_1+\lambda\Pi_2)$ is closed.
\end{proof}

As an application of Theorem \ref{lem closed range of idempotents}, we introduce the formulas for the Moore-Penrose inverse associated with a separated pair.
We recall some basic knowledge about the Moore--Penrose inverse of an operator. Suppose that $T\in \mathcal{L}(\mathscr{H},\mathscr{K})$. The
\emph{Moore--Penrose inverse} of $T$, denoted by $T^\dag$, is the
unique element $X\in \mathcal{L}(\mathscr{K},\mathscr{H})$ satisfying
\begin{equation} \label{equ:m-p inverse} TXT=T,\quad XTX=X, \quad (TX)^*=TX, \quad\mbox{and}\quad (XT)^*=XT.\end{equation}
If such an operator $T^\dag$ exists, then $T$ is said to be \emph{Moore--Penrose invertible}. It is known that $T$ is Moore--Penrose invertible if and only if $\mathscr {R}(T)$ is closed in $\mathscr{K}$ \cite[Theorem~2.2]{Xu-Sheng}, and in this case, we have
\begin{equation}\label{equ:the range and null space of M-P inverse}\mathscr {R}(T^\dag)=\mathscr {R}(T^*)\quad{\rm and } \quad \mathscr {N}(T^\dag)=\mathscr {N}(T^*).\end{equation}
In the next result, we give a formula for the Moore--Penrose inverse of operators associated with a separated pair of submodules.
\begin{proposition}\label{thm:sth concerns M-P inverse}  Let $\Pi_1, \Pi_2\in \mathcal{L}(\mathscr{E})$ be idempotents satisfying $\Pi_1\Pi_2=\Pi_2\Pi_1=0$. Then
	\begin{equation}\label{equ:sth concerns M-P inverse}(\Pi_1+\lambda \Pi_2)^\dag=(\Pi_1+\Pi_2)^\dag \left(\Pi_1+\frac{1}{\lambda} \Pi_2\right)(\Pi_1+\Pi_2)^\dag\end{equation}
	for every $\lambda\in\mathbb{C}\setminus\{0\}$.
\end{proposition}
\begin{proof}
	Due to
	$\Pi_1\Pi_2=\Pi_2\Pi_1=0$, we have
	\begin{align*}&(\Pi_1+\Pi_2)\Pi_1=\Pi_1,\quad (\Pi_1+\Pi_2)\Pi_2=\Pi_2, \\
	&(\Pi_1+\Pi_2)^*\Pi_1^*=\Pi_1^*, \quad (\Pi_1+\Pi_2)^*\Pi_2^*=\Pi_2^*.
	\end{align*}
	It follows that $\mathscr R(\Pi_1+\Pi_2)=\mathscr R(\Pi_1)+\mathscr R(\Pi_2)$ (hence $\mathscr R(\Pi_1+\Pi_2)$ is closed), and
	for $i=1,2$, \begin{align}\label{parallel-01}&(\Pi_1+\Pi_2)(\Pi_1+\Pi_2)^\dag \Pi_i=\Pi_i\quad\mbox{and}\quad \Pi_i(\Pi_1+\Pi_2)^\dag (\Pi_1+\Pi_2)=\Pi_i,
	\end{align}
	which yield
	$$\Pi_i=\Pi_i\cdot \Pi_i=\Pi_i (\Pi_1+\Pi_2)(\Pi_1+\Pi_2)^\dag \Pi_i=\Pi_i(\Pi_1+\Pi_2)^\dag \Pi_i.$$
	Hence
	\begin{equation}\label{eqPi2}
	\Pi_1(\Pi_1+\Pi_2)^\dag\Pi_1=\Pi_1, \quad \Pi_2(\Pi_1+\Pi_2)^\dag \Pi_2=\Pi_2.
	\end{equation}
	It follows from Theorem \ref{lem closed range of idempotents} that $(\Pi_1+\lambda \Pi_2)^\dag$ exists.

	Now let
	\begin{equation*}\label{equ:definition ox X-1}X=(\Pi_1+\Pi_2)^\dag \left(\Pi_1+\frac{1}{\lambda} \Pi_2\right)(\Pi_1+\Pi_2)^\dag.\end{equation*}
	For simplicity, we put
	$$\widetilde{P}=(\Pi_1+\Pi_2)(\Pi_1+\Pi_2)^\dag,\quad \widetilde{Q}=(\Pi_1+\Pi_2)^\dag(\Pi_1+\Pi_2).$$
	Then both $\widetilde{P}$ and $\widetilde{Q}$ are projections. In view of \eqref{parallel-01} and \eqref{eqPi2}, we have
	\begin{align*}
	(\Pi_1+\lambda \Pi_2)X&=\big[(\Pi_1+\Pi_2)+(\lambda-1)\Pi_2 \big] (\Pi_1+\Pi_2)^\dag \\
	&\qquad \cdot \left[(\Pi_1+\Pi_2)+\frac{1-\lambda}{\lambda} \Pi_2\right](\Pi_1+\Pi_2)^\dag\\
	&=\left[\widetilde{P}+(\lambda-1)\Pi_2 (\Pi_1+\Pi_2)^\dag \right] \left[\widetilde{P}+\frac{1-\lambda}{\lambda} \Pi_2 (\Pi_1+\Pi_2)^\dag\right]\\
	&=\widetilde{P}+\left[\lambda-1+\frac{1-\lambda}{\lambda}-\frac{(1-\lambda)^2}{\lambda}\right]\Pi_2(\Pi_1+\Pi_2)^\dag=\widetilde{P}.
	\end{align*}
	Therefore,
	$(\Pi_1+\lambda \Pi_2)X$ is self-adjoint and
	$$(\Pi_1+\lambda \Pi_2)X(\Pi_1+\lambda \Pi_2)=\widetilde{P}\Pi_1+\lambda \widetilde{P}\Pi_2=\Pi_1+\lambda \Pi_2.$$
	Similarly,
	\begin{align*}
	X(\Pi_1+\lambda \Pi_2)=\widetilde{Q},
	\end{align*}
	$X(\Pi_1+\lambda \Pi_2)$ is self-adjoint, and
	$X(\Pi_1+\lambda \Pi_2)X=\widetilde{Q}X=X$. Thus, the four conditions stated in \eqref{equ:m-p inverse} are satisfied for $\Pi_1+\lambda \Pi_2$ and $X$.
\end{proof}

\begin{remark} The group inverse case of formula \eqref{equ:sth concerns M-P inverse} was stated in \cite[Theorem~3.8]{Ando} under the condition that  $\Pi_1=P_{\mathscr{R}(F)}$ and $\Pi_2=P_{\mathscr{N}(F)}$, where $F$ is an idempotent on a Hilbert space.
\end{remark}

The following example shows that there are idempotents $\Pi_1$ and $\Pi_2$ such that the pair $\big(\mathscr R(\Pi_1),\mathscr R(\Pi_2)\big)$ is separated but $\mathscr R(\Pi_1+\Pi_2)$ is not closed.
\begin{example}Let $\mathcal K$ be a separable Hilbert space and let $\{e_i:i\in\mathbb{N}\}$ be its usual orthonormal basis. Let $U$ be the unilateral shift given by
	$Ue_i=e_{i+1}$ for $i\in\mathbb{N}$. Define $T\in \mathbb{B}(\mathcal K)$ by
	\[
	Te_i=\frac{2}{i}e_i\qquad(i\geq 1)\,.
	\]
	Put $\mathcal H=\mathcal K\oplus \mathcal K$ and set
	$$\Pi_1=\left(
	\begin{array}{cc}
	I & -T \\
	0 & 0 \\
	\end{array}
	\right),\quad \Pi_2=\left(
	\begin{array}{cc}
	I & 0 \\
	U & 0 \\
	\end{array}
	\right).$$
	A simple computation shows that both $\Pi_1$ and $\Pi_2$ are idempotents in  $\mathbb{B}(\mathcal{H})$ such that $\mathscr R(\Pi_1)\cap \mathscr R(\Pi_2)=0$. Note that $\mathscr{R}(\Pi_1)=\overline{\mathscr {R}(\Pi_1)}=\overline{\mathscr{R}(\Pi_1\Pi_1^*)}=\mathcal K\oplus 0$ and that $\mathscr{R}(\Pi_2)=\overline{\{x\oplus Ux:x\in \mathcal K\}}$, so we have
	$$\mathscr{R}(\Pi_1)+\mathscr{R}(\Pi_2)=\mathcal{K}\oplus \mathscr{R}(U),$$
	which is obviously closed in  $\mathcal{H}$. Thus, $\big(\mathscr{R}(\Pi_1),\mathscr{R}(\Pi_2)\big)$ is a separated pair of subspaces in $\mathcal H$.
	
	We claim that $\mathscr R(\Pi_1+\Pi_2)$ is not closed. To see this, let
	$$x_n=\sum\limits_{i=1}^n\frac{1}{i}e_i,\quad y_n=\sum\limits_{i=1}^n e_i$$ for each $n\in\mathbb{N}$.
	Then
	\[
	(\Pi_1+\Pi_2)(x_n\oplus y_n)=0\oplus Ux_n\to 0\oplus \sum_{i=1}^\infty\frac{1}{i}e_{i+1}:= 0\oplus \xi\,.
	\]
	We claim that $0\oplus\xi\not\in \mathscr{ R}(\Pi_1+\Pi_2)$. In fact, if there exist $x=\sum_{i=1}^\infty\alpha_ie_i,y=\sum_{i=1}^\infty\beta_ie_i\in \mathcal K$ such that $(\Pi_1+\Pi_2)(x\oplus y)=0\oplus\xi$, then
	$$2x-Ty=0,\quad Ux=\sum_{i=1}^\infty\alpha_ie_{i+1}=\sum_{i=1}^\infty\frac{1}{i}e_{i+1}.$$
	It follows that $\alpha_i=\frac{1}{i}$ and
	\begin{align*}\label{eqcount}\sum_{i=1}^\infty\frac{2\beta_i}{i}e_i= Ty=2x=\sum_{i=1}^\infty\frac{2}{i}e_{i}.\end{align*}
	Hence $\beta_i=1$ for all $i\in\mathbb{N}$, which is a contradiction since $\|y\|^2=\sum_{i=1}^\infty |\beta_i|^2<\infty$. Thus, $\mathscr R(\Pi_1+\Pi_2)$ is not closed.
\end{example}

The following example shows that the separated condition in Theorem  \ref{lem closed range of idempotents} is necessary.
%=================================================================
\begin{example}
Let $\mathscr A=C[0,1]$ be the $C^*$-algebra of all continuous complex-valued functions on $[0,1]$. Let $\mathscr E=\mathscr A\oplus \mathscr A$. Let $g\in \mathscr A$ be defined by $g(\lambda)=\lambda$ for all $\lambda\in [0,1]$. Let $T$ and $S$ be an idempotent operators in $\mathcal{L}(\mathscr E)$ defined as follows:
\[
T\begin{pmatrix}
f_1\\
f_2
\end{pmatrix}=\begin{pmatrix}
f_1\\
gf_1
\end{pmatrix}\quad\mbox{and}\quad  S\begin{pmatrix}
f_1\\
f_2
\end{pmatrix}=\begin{pmatrix}
f_1\\
-gf_1
\end{pmatrix}\,.
\]  
It is easy to verify that $\mathscr R(S)\cap \mathscr R(T)=\{0\}$. If we define $f_n\in \mathscr A$ by 
\[
f_n(\lambda)=\begin{cases}
\frac{1}{\lambda}\,,&\frac{1}{n}\leq\lambda\leq 1\,,\\
n\,,&\lambda\leq \frac{1}{n}\,,
\end{cases}
\] 
then $(T-S)\begin{pmatrix}
f_n\\
0
\end{pmatrix}=T\begin{pmatrix}
f_n\\
0
\end{pmatrix}+S\begin{pmatrix}
-f_n\\
0
\end{pmatrix}$ tends to $\begin{pmatrix}
0\\
2
\end{pmatrix}\not \in \mathscr R(T)+\mathscr R(S)\supseteq\mathscr R(T-S)$. This shows that  $(\mathscr R(T),\mathscr R(S))$ is not a separated pair of submodules and $\mathscr R(T-S)$ is not closed. It is easy to see that $\mathscr R(T+S)= \mathscr A\oplus 0$.
\end{example}

%======================================================%
Let $\Pi\in \mathcal{L}(\mathscr{E})$ be an idempotent, and let $P$ and $Q$ be projections  from $\mathscr{E}$ onto $\mathscr {R}(\Pi)$ and $\mathscr {N}(\Pi)$, respectively. By Lemma~\ref{lem:norm of two projections less than one} we have $\|PQ\|<1$, and it is shown  in \cite[Theorem~1.3]{Koliha} that
\[
\Pi=(I-PQ)^{-1}P(I-PQ)\,.
\]

Inspired by the above observation, we give an alternative description of the idempotents $\Pi_1$ and $\Pi_2$ constructed in the proof of Theorem~\ref{th equivalent of separation}.
%======================================================%

\begin{lemma}\label{lem:idempotents induced by complemented projections} Let $(\mathscr H,\mathscr K)$ be a separated pair of orthogonally complemented closed submodules of $\mathscr E$.  Let
	\begin{equation}\label{equ:idempotents induced by complemented projections}\Pi_{P,Q}=(I-PQ)^{-1}P(I-PQ)\quad \mbox{and}\quad \Pi_{Q,P}=(I-QP)^{-1}Q(I-QP)\,,\end{equation}
	where $P$ and $Q$ are the projections from $\mathscr{E}$ on $\mathscr H$ and $\mathscr K$, respectively. Then $\Pi_{P,Q}$ and $\Pi_{Q,P}$ are idempotents such that
	\begin{align}\label{ali:induced range and null space-1}&\mathscr {R}(\Pi_{P,Q})=\mathscr {R}(P), \quad \mathscr{N}(\Pi_{P,Q})=\mathscr{R}(Q)+\mathscr{N}(P)\cap\mathscr {N}(Q),\\
	\label{ali:induced range and null space-2}&\mathscr {R}(\Pi_{Q,P})=\mathscr {R}(Q), \quad \mathscr {N}(\Pi_{Q,P})=\mathscr {R}(P)+\mathscr {N}(P)\cap\mathscr {N}(Q).
	\end{align}
\end{lemma}
\begin{proof} It follows from Lemma~\ref{lem:norm of two projections less than one} that
	$$\|QP\|=\|PQ\|<1,$$
	so $I-PQ$ and $I-QP$ are both invertible. Evidently, both $\Pi_{P,Q}$ and $\Pi_{Q,P}$ are idempotents.
	
	It is obvious that $(I-PQ)^{-1}P=P(I-QP)^{-1}$,
	which gives 	$\mathscr {R}(\Pi_{P,Q})\subseteq \mathscr {R}(P)$. Also, by \eqref{equ:idempotents induced by complemented projections}, we have
	\begin{align}\label{eqPiQ}
	\Pi_{P,Q} P=(I-PQ)^{-1}(P-PQP)=(I-PQ)^{-1}(I-PQ)P=P,
	\end{align}
	and therefore $\mathscr {R}(P)\subseteq\mathscr {R}(\Pi_{P,Q})$. This shows that $\mathscr {R}(\Pi_{P,Q})=\mathscr {R}(P)$.
	
	The idempotent $\Pi_{P,Q}$ can be rewritten as
	$$\Pi_{P,Q}=(I-PQ)^{-1}P(I-Q),$$
	which gives $\mathscr {R}(Q)\subseteq \mathscr {N}(\Pi_{P,Q})$. Furthermore, the equation above indicates that for every $u\in\mathscr {N}(Q)$, we have
	$u\in \mathscr {N}(\Pi_{P,Q})$ if and only if $u\in\mathscr {N}(P)$. This completes the proof of the second equality in \eqref{ali:induced range and null space-1}. Clearly, the equations in \eqref{ali:induced range and null space-2} can be derived directly from \eqref{ali:induced range and null space-1} by exchanging $P$ with $Q$.
\end{proof}

%======================================================%

\begin{lemma}\label{lem:an auxiliary lemma-1}{\rm \cite[Lemma~3.1]{QXZ}} Let $\Pi_1,\Pi_2\in\mathcal{L}(\mathscr{E})$ be idempotents. If $\mathscr{R}(\Pi_2)\subseteq \mathscr{R}(\Pi_1)$ and
	$\mathscr{N}(\Pi_2)\subseteq \mathscr{N}(\Pi_1)$, then $\Pi_1=\Pi_2$.
\end{lemma}

%======================================================%
\begin{theorem}\label{cor:the equivalent condition of pi1 add pi2 equals I} Suppose that $P, Q\in \mathcal{L}(\mathscr {E})$
	are projections such that $\big(\mathscr{R}(P),\mathscr{R}(Q)\big)$ is separated. Let $\Pi_1$ and $\Pi_2$ be idempotents such that
	$\mathscr{R}(\Pi_1)=\mathscr{R}(P)$ and $\mathscr{R}(\Pi_2)=\mathscr R(Q)$, respectively. Let $\widetilde{P}$ be the projection from $\mathscr{E}$ onto $\mathscr R(P)+\mathscr R(Q)$. Then the following statements are equivalent:
	\begin{enumerate}
		\item[{\rm (i)}] $\Pi_1+\Pi_2=\widetilde{P}$;
		\item[{\rm (ii)}] $\Pi_{P,Q}=\Pi_1$ and $\Pi_{Q,P}=\Pi_2$.
	\end{enumerate}
\end{theorem}
\begin{proof}
	Note that by \eqref{ali:induced range and null space-1}, we have
	$\mathscr{R}(\Pi_{P,Q})=\mathscr{R}(P)=\mathscr{R}(\Pi_1)$ and
	$$\mathscr{N}(\Pi_{P,Q})=\mathscr{R}(\Pi_2)+\mathscr{N}(P)\cap\mathscr{N}(Q).$$
	Therefore, by Lemma~\ref{lem:an auxiliary lemma-1}, we observe that
	$\Pi_{P,Q}=\Pi_1$ if and only if
	\begin{equation}\label{eqNullPi1}
	\mathscr{N}(\Pi_1)\subseteq \mathscr{R}(\Pi_2)+\mathscr{N}(P)\cap \mathscr{N}(Q).
	\end{equation}
	In a similar way, we conclude that
	$\Pi_{Q,P}=\Pi_2$ if and only if
	\begin{equation}\label{eqNullPi2}
	\mathscr{N}(\Pi_2)\subseteq\mathscr{R}(\Pi_1)+\mathscr{N}(P)\cap \mathscr{N}(Q).
	\end{equation}
	
	(i)$\Longrightarrow$(ii). Suppose that $\Pi_1+\Pi_2=\widetilde{P}$. Then we have
	$$\mathscr{N}(\Pi_1)=\mathscr{N}(\widetilde{P}-\Pi_2)\subseteq \mathscr{R}(\Pi_2)+\mathscr{N}(P)\cap \mathscr{N}(Q),$$
	$$\mathscr{N}(\Pi_2)=\mathscr{N}(\widetilde{P}-\Pi_1)\subseteq \mathscr{R}(\Pi_1)+\mathscr{N}(P)\cap \mathscr{N}(Q),$$
	which indicate the validity of \eqref{eqNullPi1} and \eqref{eqNullPi2}, and therefore, we have $\Pi_{P,Q}=\Pi_1$ and $\Pi_{Q,P}=\Pi_2$.
	
	(ii)$\Longrightarrow$(i). We show that $\Pi_{P,Q}+\Pi_{Q,P}=\widetilde{P}$. As a matter of fact, by a straightforward calculation as in \eqref{eqPiQ}, we have
	\[
	\Pi_{P,Q}P=P\,,\quad\Pi_{Q,P}Q=Q\,,\quad \Pi_{P,Q}Q=0=\Pi_{Q,P}P\,.
	\]
	Thus
	\begin{align}\label{eqP and Q}
	(\Pi_{P,Q}+\Pi_{Q,P})(Px+Qy)=Px+Qy
	\end{align}
	for each $x,y\in \mathcal E$. Moreover,
	\begin{align}\label{eqP and Q2}
	(\Pi_{P,Q}+\Pi_{Q,P})z=0, \qquad (z\in \mathscr{N}(P)\cap\mathscr{N}(Q))\,.
	\end{align}
	Since $\mathscr{R}(P)+\mathscr{R}(Q)$ is orthogonally complemented, employing \eqref{eqP and Q} and \eqref{eqP and Q2} we get the desired result.
\end{proof}

%======================================================%
%======================================================%
%======================================================%
\section{ Concordant pairs of closed submodules in terms of the states}
\label{section HC}
For a positive linear functional $f$ on $\mathscr A$, we set $$\mathscr N_f=\{x\in \mathscr E:f\left(\langle x,x\rangle\right)=0\}.$$
It follows from the Cauchy--Schwarz inequality that
$$\mathscr N_f=\{x\in \mathscr E:f\left(\langle y,x\rangle\right)=f\left(\langle x,y\rangle\right)=0, \mbox{~for all~} y\in \mathscr E\}.$$
Therefore, $\mathscr N_f$ is a closed subspace of $\mathscr E$, and the quotient space $\mathscr E/\mathscr N_f$ is a pre-Hilbert space equipped with the inner product $\langle\cdot,\cdot\rangle_f$ defined by
\[
\langle x+\mathscr N_f,y+\mathscr N_f\rangle_f=f\big(\langle x,y\rangle\big).
\]
Let $\mathscr E_f$ be the  completion of $\mathscr E/\mathscr N_f$. Let $\iota_f:\mathscr E\to\mathscr E_f $ be the natural map, that is, $\iota_f(x)=x+\mathscr N_f$.  If $\mathscr H$ is a closed submodule of $\mathscr E$, then we consider $\mathscr H_f$ as the closure of $\iota_f(\mathscr H)=\{x+\mathscr N_f:x\in \mathscr H\}$. That is, $\mathscr H_f=\overline{\{x+\mathscr N_f:x\in \mathscr H\}}$, see \cite{Pa} for more information. We have the following Theorem.
\begin{theorem}\label{th there is a state}{\cite[Theorem 3.1]{Kaad}}
	Let $\mathscr L\subseteq \mathscr E$ be a closed convex subset of the Hilbert $C^*$-module $\mathscr E$ over $\mathscr A$. For each vector $ x_0\in \mathscr E\backslash \mathscr L$ there exists a state $f$ on $\mathscr A$ such that $\iota_f(x_0)$ is not in the closure of $\iota_f(\mathscr L)$.
	In particular, there exists a state $f$ such that $\iota_f(\mathscr L)$ is not dense in $\mathscr E_f$ and thus, when $\mathscr L$ is a
	submodule, $\iota_f(L)^\perp\neq0$.
\end{theorem}
The above theorem, \cite[Corollary 1.17]{Pi} and the paragraph after \cite[Proposition 1.6]{Bram}, pint out the following lemma.

\begin{lemma}\label{lemma eq}
	Let $\mathscr H$ and $\mathscr K$ be closed submodules of $\mathscr E$. Then $\mathscr H=\mathscr K$ if and only if $\mathscr H_f=\mathscr K_f$ for each $f\in \mathrm{S}(\mathscr A)$, if and only if $\mathscr H_f=\mathscr K_f$ for each $f\in \mathrm{PS}(\mathscr A)$.
\end{lemma}
Employing Theorem \ref{th there is a state} and \cite[Theorem 2.1]{Kaad2} to get  equivalences for orthogonally complemented submodules, see also \cite[Proposition 1.6]{Bram}.
\begin{proposition}\label{Thorthogonaliy condition}
	Let $\mathscr H$ be a closed submodule of $\mathscr E$. Then the following statements are equivalent:
	\begin{itemize}
		\item[(i)]
		$\mathscr H$
		is orthogonally complemented in  $\mathscr E$.
		\item[(ii)]
		$
		(\mathscr H_f)^\perp=(\mathscr H^\perp)_f$ for each $f\in \mathrm{S}(\mathscr A).
		$
		\item[(iii)]
		$
		(\mathscr H_f)^\perp=(\mathscr H^\perp)_f$ for each $f\in \mathrm{PS}(\mathscr A).
		$
	\end{itemize}
\end{proposition}
%%%%%%%%%%%%%%%%%%%%%%%%%%%%%%%%%%%%%%%%%%%%%%%%%%%%%%%%%%%%%%%55

%%%%%%%%%%%%%%%%%%%%%%%%%%%%%%%%%%%%%%%%%%%%%%%%%%%%%%%%%%%%%%%%%%%%

\begin{definition}\label{def:concordant}	
	The pair $(\mathscr H,\mathscr K)$ of closed submodules of $\mathscr E$ is said to be \emph{concordant}  if $\overline{\mathscr H^\perp + \mathscr K^\perp}$ is orthogonally complemented in $\mathscr E$, that is, $\mathscr E$ can be decomposed orthogonally as
	\begin{equation}\label{equ:orth of E-concordant}\mathscr E = (\mathscr H\cap\mathscr K)\oplus \overline{\mathscr H^\perp + \mathscr K^\perp}.\end{equation}
\end{definition}

\begin{theorem}\label{th concorant}
	Let $\mathscr H$ and $\mathscr K$ be closed submodules of $\mathscr E$. Then the following statements are equivalent:
	\begin{itemize}
\item[(i)]  The pair $(\mathscr H,\mathscr K)$  is concordant.
\item[(ii)] For every $f\in \mathrm{S}(\mathscr A)$,
\begin{equation}\label{pure state}
(\mathscr H\cap\mathscr K)_f=\big((\mathscr H^\perp)_f\big)^\perp\cap \big((\mathscr K^\perp)_f\big)^\perp\,.
\end{equation}
\item[(iii)] For every $f\in {\mathrm PS}(\mathscr A)$,
\begin{equation}\label{pure state 1}
(\mathscr H\cap\mathscr K)_f=\big((\mathscr H^\perp)_f\big)^\perp\cap \big((\mathscr K^\perp)_f\big)^\perp\,.
\end{equation}
	\end{itemize} 
\end{theorem}
\begin{proof}
	(i)$\Longrightarrow$(ii). Suppose that $(\mathscr H, \mathscr K)$ is  concordant. Given $f\in \mathrm{S}(\mathscr A)$, let $X\subseteq \mathscr E/\mathscr N_f$, $Y\subseteq \mathscr E$ and
	$Z\subseteq \mathscr E_f$ be defined respectively by
	$$X=\big\{x+\mathscr N_f: x\in \mathscr H^\perp+\mathscr K^\perp\big\},\ Y=\overline{\mathscr H^\perp+\mathscr K^\perp},\ Z=\big((\mathscr H^\perp)_f\big)^\perp\cap \big((\mathscr K^\perp)_f\big)^\perp.$$
	It is clear that
	\begin{equation}\label{eqsub}
	Y^\perp=\mathscr H\cap\mathscr K,\quad Y_f= \overline{X}\quad\mbox{and}\quad X^\perp=Z.
	\end{equation}
	Making a use of Theorem \ref{Thorthogonaliy condition} to $Y$,  we get
	\begin{align*}(\mathscr H\cap\mathscr K)_f=(Y^\perp)_f=(Y_f)^\perp=\overline{X}^\perp=X^\perp=Z
	\end{align*}
	in the Hilbert space $\mathscr{E}_f$. This shows the validity of \eqref{pure state}.

(ii)$\Longrightarrow$(iii). It is clear.\\
	(iii)$\Longrightarrow$(i).   Given
	$f\in {\mathrm PS}(\mathscr A)$, let $X$, $Y$ and $Z$ be defined as above. Then
	\begin{align}\label{eq121}
	\big[(\mathscr H\cap\mathscr K)_f\big]^\perp=Z^\perp=\overline{X}=Y_f\subseteq \big[(\mathscr H\cap\mathscr K)^\perp\big]_f\subseteq \big[(\mathscr H\cap\mathscr K)_f\big]^\perp,
	\end{align}
	which gives
	$$\big[(\mathscr H\cap\mathscr K)_f\big]^\perp=\big[(\mathscr H\cap\mathscr K)^\perp\big]_f.$$
	From Theorem \ref{Thorthogonaliy condition} we  conclude that $\mathscr H\cap\mathscr K$ is orthogonally complemented in $\mathscr E$.
	Furthermore, from \eqref{eq121} we  obtain
	$$Y_f=\big[(\mathscr H\cap\mathscr K)^\perp\big]_f,\quad (f\in {\mathrm PS}(\mathscr A)),$$
	whence,  by Lemma \ref{lemma eq},  $Y=(\mathscr H\cap\mathscr K)^\perp$.
\end{proof}
It is remarkable if the pair $(\mathscr H,\mathscr K)$ of closed submodules of $\mathscr E$  is concordant, then \eqref{pure state} gives 
\[
\mathscr H_f\cap \mathscr K_f\subseteq \big((\mathscr H^\perp)_f\big)^\perp\cap \big((\mathscr K^\perp)_f\big)^\perp=(\mathscr H\cap \mathscr K)_f\subseteq \mathscr H_f\cap \mathscr K_f\,.
\]
This is a corollary as follows: 
%======================================================================
\begin{corollary}\label{concordant condition}
	Let the pair $(\mathscr H,\mathscr K)$ of closed submodules of $\mathscr E$  be concordant. Then 
	\[
	(\mathscr H\cap \mathscr K)_f=\mathscr H_f\cap \mathscr K_f\qquad(f\in {\mathrm S}(\mathscr A))\,. 
	\]
\end{corollary}
The following example shows that the reverse of Corollary \ref{concordant condition} is not valid.
\begin{example}
Let $\mathscr{A}=C[0,1]$ and $\mathscr E=\mathscr A$. Let $\mathscr H$ and $\mathscr K$ be closed submodules of $\mathscr E$ defined as follows:
\[
\mathscr H=\{\tau \in \mathscr A: \tau(0)=0\}\quad \mbox{and}\quad \mathscr K=\{\tau\in \mathscr A:\tau(1)=0 \}\,.
\]  
Then 
\[
\mathscr H^\perp=\mathscr K^\perp=\{0\}\quad \mbox{and}\quad \mathscr H\cap \mathscr K=\{\tau \in \mathscr A: \tau(0)=\tau(1)=0\}\,.
\] 
Hence $(\mathscr H,\mathscr K)$ is not concordant. Let $f$ be an arbitrary  pure state on $\mathscr A$. Then there exists $ t_0\in [0,1]$ such that $f(\tau)=\tau(t_0)$ for all $\tau\in \mathscr A$. It is easy to check that $(\mathscr H\cap \mathscr K)_f=\mathscr H_f\cap \mathscr K_f$.
\end{example} 
%%%===================================================================================

%%%%===========================================================================
Let  $\mathscr H$ and $\mathscr K$ be closed submodules of $\mathscr E$. As in the Hilbert space case, we define  the cosine of the Friedrichs angle between $\mathscr H$ and $\mathscr K$ by 
\begin{equation}\label{definiatin of cH K}
	c(\mathscr H,\mathscr K)=\sup\{\|\langle x,y\rangle\|:\|x\|=\|y\|=1,x\in \mathscr H\cap (\mathscr H\cap \mathscr K)^\perp, y\in \mathscr K\cap (\mathscr H\cap \mathscr K)^\perp\}
\end{equation}
  Let  $\mathscr H, \mathscr K$, and $\mathscr H\cap \mathscr K$ be  orthogonally complemented in $\mathscr E$.  Then in \cite{Luo-Moslehian-Xu}, $c(\mathscr H,\mathscr K)$ 
is formulated by
\begin{equation}\label{equ:formula for c H K}
c(\mathscr H,\mathscr K) = \|P_\mathscr HP_\mathscr K(I - P_{\mathscr H\cap\mathscr K})\| =\|P_\mathscr HP_\mathscr K-P_{\mathscr H\cap \mathscr K}\|\,.
\end{equation}
\begin{definition}
	Let $\mathscr H$ and $\mathscr K$ be closed submodules of $\mathscr E$. We define the {\it  cosine of the local Friedrichs angle} between $\mathscr H$ and $\mathscr K$ by 
	\[
	\alpha(\mathscr H,\mathscr K):=\sup_{f\in \mathrm{S}(\mathscr A)} c(\mathscr H_f,\mathscr K_f)\,.
	\]
	Also, we define the {\it cosine of the local Dixmier angle} between $\mathscr H$ and $\mathscr K$ by 
	\[
	\alpha_0(\mathscr H,\mathscr K):=\sup_{f\in \mathrm{S}(\mathscr A)} c_0(\mathscr H_f,\mathscr K_f)\,.
	\]
\end{definition}

\begin{remark}
	Let  $\mathscr H$ and $\mathscr K$ be orthogonally complemented submodules of $\mathscr E$ such that $(\mathscr H, \mathscr K)$ is concordant. It follows from \cite[Corollary 3.4]{Bram} that
	\begin{equation*}
	\alpha(\mathscr H,\mathscr K)=c(\mathscr H,\mathscr K)\,,
	\end{equation*}
\end{remark}

\begin{theorem}\label{Friedrichs angle}
	Let  $\mathscr H$ and $\mathscr K$ be closed submodules of $\mathscr E$ such that $(\mathscr H, \mathscr K)$ is concordant  and let $\overline{\mathscr H+\mathscr K}$ be an orthogonally complemented submodule. Then
	\begin{equation*}
	\alpha(\mathscr H,\mathscr K)=\alpha(\mathscr H^\perp,\mathscr K^\perp)\,,
	\end{equation*}
\end{theorem}
\begin{proof}
	Since   $(\mathscr H,\mathscr K)$ is concordant,  Theorems~\ref{th concorant} and Corollary \ref{concordant condition} ensure that
	$$\big((\mathscr H^\perp)_f\big)^\perp\cap \big((\mathscr K^\perp)_f\big)^\perp=(\mathscr H\cap \mathscr K )_f=\mathscr H_f\cap\mathscr K_f,\quad ( f\in \mathrm{S}(\mathscr A)).$$
	It follows from \cite[Lemma 10, Theorem 16]{Deutsch} that for every $f\in \mathrm{S}(\mathscr A)$,
	\begin{align*}
	c(\mathscr (\mathscr H^\perp)_f,\mathscr (\mathscr K^\perp)_f)&=c( (\mathscr H^\perp)_f^\perp, (\mathscr K^\perp)_f^\perp)\\
	&=\|P_{(\mathscr H^\perp)_f^\perp}P_{(\mathscr K^\perp)_f^\perp}-P_{(\mathscr H^\perp)_f^\perp\cap(\mathscr K^\perp)_f^\perp}\|\\
	&=\|P_{(\mathscr H^\perp)_f^\perp}P_{(\mathscr K^\perp)_f^\perp}-P_{\mathscr H_f\cap\mathscr K_f}\|\\
	&=\|P_{(\mathscr H^\perp)_f^\perp}P_{(\mathscr K^\perp)_f^\perp}P_{(\mathscr H\cap\mathscr K)_f^\perp}\|\\
	&\geq \|P_{\mathscr H_f}P_{(\mathscr K^\perp)_f^\perp}P_{(\mathscr H\cap\mathscr K)_f^\perp}\|\qquad({\rm since ~}P_{(\mathscr H^\perp)_f^\perp}\geq P_{\mathscr H_f})\\ 
	&= \|P_{\mathscr H_f}P_{(\mathscr H\cap\mathscr K)_f^\perp}P_{(\mathscr K^\perp)_f^\perp}\|\\ 
	&\geq \|P_{\mathscr H_f}P_{(\mathscr H\cap\mathscr K)_f^\perp}P_{\mathscr K_f}\|=\|P_{\mathscr H_f}P_{\mathscr K_f}P_{(\mathscr H\cap\mathscr K)_f^\perp}\|=c(\mathscr H_f,\mathscr K_f)\,,
	\end{align*}
	which gives
	\begin{equation}\label{equ:norm-1 for f}
	c(\mathscr (\mathscr H^\perp)_f,\mathscr (\mathscr K^\perp)_f)\geq c(\mathscr H_f,\mathscr K_f)\qquad(f\in \mathrm{S}(\mathscr A))\,.
	\end{equation}
Therefore,
		\begin{equation}\label{equ:norm-1 for alpha}
	\alpha(\mathscr  H^\perp,\mathscr K^\perp)\geq \alpha(\mathscr H,\mathscr K)\,.
	\end{equation}
	Since $\overline{(\mathscr H+\mathscr K)}\oplus (\mathscr H^\perp\cap \mathscr K^\perp)=\mathscr E$, we have $\overline{(\mathscr H^{\perp\perp}+\mathscr K^{\perp\perp})}\oplus (\mathscr H^\perp\cap \mathscr K^\perp)=\mathscr E$. This entails that $(\mathscr H^\perp,\mathscr K^\perp)$ is concordant. Hence,
	\begin{equation}\label{concordant perp}
(\mathscr H^\perp)_f\cap (\mathscr K^\perp)_f=(\mathscr H^\perp\cap \mathscr K^\perp)_f=\big((\mathscr H+\mathscr K)^\perp\big)_f=\big((\mathscr H+\mathscr K)_f\big)^\perp= \mathscr H_f^\perp\cap\mathscr K_f^\perp\,,
	\end{equation}
	for each state $f$ on $\mathscr A$.
	  Now, if we set $H_f=(\mathscr H^\perp)_f\cap ((\mathscr H^\perp)_f\cap (\mathscr K^\perp)_f)^\perp$, then we have 
	\[
	H_f\subseteq (\mathscr H_f)^\perp\cap \big((\mathscr H^\perp)_f\cap (\mathscr K^\perp)_f\big)^\perp:= H'_f\,.
	\]
	By the same reasoning, $K_f\subseteq K'_f$.
	Hence,
	\begin{align*}
	c((\mathscr H^\perp)_f,(\mathscr K^\perp)_f)
	&=\sup\{\|\langle \tilde{u},\tilde{v}\rangle\|;  \tilde{u}\in H_f, \| \tilde{u}\|\leq 1, \tilde{v}\in K_f,\| \tilde{v}\|\leq 1 \}\\
	&\leq \sup\{\|\langle \tilde{u},\tilde{v}\rangle\|;  \tilde{u}\in H'_f, \| \tilde{u}\|\leq 1, \tilde{v}\in K'_f,\| \tilde{v}\|\leq 1 \}\\
		&=c(\mathscr H_f^\perp,\mathscr K_f^\perp)\qquad({{\rm by~}\eqref{concordant perp}})\\
&=c(\mathscr H_f,\mathscr K_f)
	\end{align*}
	Thus $c(\mathscr H_f,\mathscr K_f)=c((\mathscr H^\perp)_f,(\mathscr K^\perp)_f)$, and hence, $\alpha(\mathscr H,\mathscr K)=\alpha(\mathscr H^\perp,\mathscr K^\perp)$
\end{proof}
%============================================================================
We use the following lemma in the next theorem.
\begin{lemma}\cite[Lemma 5,Lemma 9]{Deutsch}\label{lemma projection}
Let $\mathscr M$ and $\mathscr N$ be closed subspaces of a Hilbert space $\mathscr H$. Then the following statements are equivalent:
\begin{itemize}
\item[(i)] $P_\mathscr M$ and $P_\mathscr N$ commute.
\item[(ii)] $P_\mathscr MP_\mathscr N=P_{\mathscr M\cap\mathscr N}$.
\item[(iii)] $P_{\mathscr M^\perp}$ and $P_{\mathscr N^\perp}$ commute.
\item[(iv)] $P_{\mathscr M^\perp}$ and $P_{\mathscr N}$ commute.
\item[(v)] $P_\mathscr M$ and $P_{\mathscr N^\perp}$ commute.
\item[(vi)] $\mathscr M=(\mathscr M\cap \mathscr N)+(\mathscr M\cap \mathscr N^\perp)$.
 	\end{itemize}
\end{lemma}
%===========================================================================
\begin{theorem}
	Let  $\mathscr H$ and $\mathscr K$ be closed submodules of $\mathscr E$ such that $(\mathscr H, \mathscr K)$ and $(\mathscr H, \mathscr K^\perp)$ are concordant  and let $\overline{\mathscr H+\mathscr K}$ be an orthogonally complemented submodule. Then the following statements are equivalent:
\begin{itemize}
\item[(i)] $\alpha(\mathscr H, \mathscr K)=0$.
\item[(ii)] $\mathscr H=(\mathscr H\cap \mathscr K)+(\mathscr H\cap \mathscr K^
\perp)$.
\end{itemize}
\end{theorem}
\begin{proof}
(i)$\Rightarrow$(ii). Let $f\in \mathrm{S}(\mathscr A)$. Since $\overline{\mathscr H+\mathscr K}$ is  orthogonally complemented, we have  $\mathscr E=(\overline{\mathscr H+\mathscr K})\oplus(\mathscr H^\perp\cap\mathscr K^\perp)$. Hence,
\[
\mathscr E_f=(\overline{\mathscr H+\mathscr K})_f\oplus(\mathscr H^\perp\cap \mathscr K^\perp)_f\subseteq (\overline{\mathscr H+\mathscr K})_f+[(\mathscr H_f)^\perp\cap (\mathscr K_f)^\perp].
\]
It is clear that $(\overline{\mathscr H+\mathscr K})_f$ is orthogonal to $(\mathscr H_f)^\perp\cap (\mathscr K_f)^\perp$, hence 
\begin{equation}\label{eq equiation of projection}
(\mathscr H^\perp\cap \mathscr K^\perp)_f=(\mathscr H_f)^\perp\cap (\mathscr K_f)^\perp\,.
\end{equation}
It follows from Theorem \ref{th concorant} and Corollary \ref{concordant condition} that $\overline{(\mathscr H^\perp)_f+(\mathscr K^\perp)_f}=(\mathscr H_f)^\perp+(\mathscr K_f)^\perp$. Thus 
\begin{align*}
P_{(\mathscr H_f)^\perp}+ P_{(\mathscr K_f)^\perp}-P_{(\mathscr H^\perp\cap \mathscr K^\perp)_f}&=
 P_{(\mathscr H_f)^\perp}+ P_{(\mathscr K_f)^\perp}-P_{(\mathscr H_f)^\perp\cap (\mathscr K_f)^\perp}\quad(\mbox{by ~\eqref{eq equiation of projection}})\\
 &=P_{(\mathscr H_f)^\perp+(\mathscr K_f)^\perp}\\
 &=P_{\overline{(\mathscr H^\perp)_f+(\mathscr K^\perp)_f}}
\end{align*}
By multiplying the both sides of above equality from left and right with $P_{(\mathscr H^\perp)_f}$ and  $P_{(\mathscr K^\perp)_f}$, respectively, we get 
\begin{equation}\label{eq 1}
P_{(\mathscr H^\perp)_f}P_{(\mathscr K^\perp)_f}-P_{(\mathscr H^\perp\cap \mathscr K^\perp)_f}=0
\end{equation}

From \eqref{eq 1} and \cite[Lemma 5]{Deutsch} we get $P_{(\mathscr K^\perp)_f}P_{(\mathscr H^\perp)_f}=P_{(\mathscr H^\perp)_f}P_{(\mathscr K^\perp)_f}$. Thus,
\begin{align*}
\mathscr H_f&\subseteq (\mathscr H^\perp)_f^\perp\\
&=\big[(\mathscr H^\perp)_f^\perp\cap (\mathscr K^\perp)_f^\perp\big]+\big[(\mathscr H^\perp)_f^\perp\cap (\mathscr K^\perp)_f\big]\quad(\mbox{by~Lemma \ref{lemma projection}})\\
&\subseteq (\mathscr H\cap \mathscr K)_f+(\mathscr H^\perp)_f^\perp\cap (\mathscr K^{\perp\perp})_f^\perp\\
&=(\mathscr H\cap \mathscr K)_f+(\mathscr H\cap \mathscr K^\perp)_f\\
&=\big((\mathscr H\cap \mathscr K)+(\mathscr H\cap \mathscr K^\perp)\big)_f\\
&\subseteq \mathscr H_f\,
\end{align*}
Therefore, by Lemma \ref{lemma eq}, we conclude that $\mathscr H=(\mathscr H\cap \mathscr K)+(\mathscr H\cap \mathscr K^\perp)$.\\
(ii)$\Rightarrow$(i). Let $f\in \mathrm{S}(\mathscr A)$ be arbitrary. Thus,
\begin{align*}
\mathscr H_f&=(\mathscr H\cap \mathscr K)_f+(\mathscr H\cap \mathscr K^\perp)_f\\
&\subseteq(\mathscr H_f\cap \mathscr K_f)+(\mathscr H_f\cap(\mathscr K_f)^\perp)\\
&\subseteq \mathscr H_f\,.
\end{align*}
This shows that $P_{\mathscr H_f}$ and $P_{\mathscr K_f}$ commute. So $c(\mathscr H_f,\mathscr K_f)=0$. Since $f$ is arbitrary, we conclude that $\alpha(\mathscr H,\mathscr K)=0$.  
\end{proof}
%==========================================================================
The following example shows that $\alpha(\mathscr H,\mathscr K)$ is different from $c(\mathscr H,\mathscr K)$ if $\mathscr H$ and $\mathscr K$ are closed submodules( not orthogonally complemented) in $\mathscr E$ such that $(\mathscr H, \mathscr K)$ is concordant.
\begin{example}
	Let $X=[-2,-1]\cup [0,1]$. Let $\mathscr A=\mathrm{C}(X)$ and let $\mathscr E=\mathscr A$. Let 
	\[
	\mathscr H=\{\tau\in \mathscr E: \tau|_{[0, \frac{2}{3}]}=0\}\quad \mbox{and}\quad \mathscr K=\{\tau\in \mathscr E: \tau|_{[\frac{1}{3},1]}=0\} \,.
	\]
	Then 
	\[
	\mathscr H^\perp=\{\tau\in \mathscr E: \tau|_{[-2,-1]\cup[ \frac{2}{3},1]}=0\}\quad \mbox{and}\quad \mathscr K^\perp=\{\tau\in \mathscr E: \tau|_{[-2,-1]\cup[0, \frac{1}{3}]}=0\} \,.
	\]
	Note that
	\begin{equation}\label{eq c[0,1]}
\mathscr H\cap \mathscr K=\{\tau\in \mathscr E:\tau|_{[0,1]}=0\}\,.
	\end{equation}
	  This implies that   $(\mathscr H\cap \mathscr K)+(\mathscr H^\perp+\mathscr K^\perp)=\mathscr E$. Hence, $(\mathscr H,\mathscr K)$ is a concordant pair of closed submodules of $\mathscr E$. Let $f$ be a  state on $\mathscr A$.  Since $(\mathscr H,\mathscr K)$ is a concordant pair, by Proposition \ref{Thorthogonaliy condition}, Corollary \ref{concordant condition}, and \eqref{eq c[0,1]}, we have 
	  \begin{align*}
	  (\mathscr H_f\cap \mathscr K_f)^\perp=[(\mathscr H\cap \mathscr K)_f]^\perp
	  =[(\mathscr M\cap \mathscr N)^\perp]_f
	  =\overline{\{\tau+\mathscr N_f;\tau\in \mathscr A, \tau|_{[-2,-1]}=0\}}\,.
	  \end{align*}
	If we set $H_f=\mathscr H_f\cap (\mathscr H_f\cap \mathscr K_f)^\perp$ and $K_f=\mathscr K_f\cap(\mathscr H_f\cap\mathscr K_f)^\perp$, then by employing \eqref{definiatin of cH K} in the Hilbert space case, we get 
	\begin{align*}
	c(\mathscr H_f,\mathscr K_f)&=\sup \{|f\langle \tau,\sigma\rangle|;\tau+\mathscr N_f\in H_f, \sigma+\mathscr N_f\in K_f\}\\
	&= \sup \{|f(\tau\bar{\sigma})|;\tau+\mathscr N_f\in H_f, \sigma+\mathscr N_f\in K_f\}=0\,.
	\end{align*}
Hence, $\alpha(\mathscr H,\mathscr K)=0$. Note that $c(\mathscr H,\mathscr K)=1$.
\end{example}
\begin{theorem}\label{thm:local char of separated pair}
Let $\mathscr H$ and $\mathscr K$ be closed submodules of $\mathscr E$ such that  $(\mathscr H^\perp,\mathscr K^\perp)$ is concordant. Then  $(\mathscr H^{\perp\perp},\mathscr K^{\perp\perp})$ is a separated pair if $\alpha_{0}(\mathscr H^{\perp\perp},\mathscr K^{\perp\perp})<1$. 
\end{theorem}
\begin{proof}
	Suppose that $\alpha_0:=\alpha_{0}(\mathscr H^{\perp\perp},\mathscr K^{\perp\perp})<1$. Then for each $f\in \mathrm{S}(\mathscr A)$ we have 
	\begin{equation}\label{eq alpha0}
|f\langle x,y\rangle|\leq \alpha_{0}f\langle x,x\rangle ^\frac{1}{2}f\langle y,y\rangle ^\frac{1}{2} \qquad(x\in \mathscr H,y\in \mathscr K)\,.
	\end{equation}
	There is a state  $f_0\in \mathrm{S}(\mathscr A)$ such that $f_0\langle x,x\rangle=\|x\|^2$. By \eqref{eq alpha0},  we have 
	\begin{align}\label{eq 02}
	\|x+y\|^2\geq f_0\langle x+y, x+y\rangle&=f_0\langle x,x\rangle +f_0\langle y,y\rangle +2{\mathrm Re}f_0\langle x,y\rangle\nonumber\\
	&\geq  \|x\|^2 +f_0\langle y,y\rangle- 2|f_0\langle x,y\rangle|\nonumber\\
	&\geq  \|x\|^2 +f_0\langle y,y\rangle-2\alpha_{0}f_0\langle x,x\rangle ^\frac{1}{2}f_0\langle y,y\rangle ^\frac{1}{2}\nonumber\\
	&=(\|x\|-f_0\langle y,y\rangle^\frac{1}{2})^2+2(1-\alpha_0)\|x\|f_0\langle y,y\rangle^\frac{1}{2}
	\end{align}
	for each $x\in \mathscr H^{\perp\perp}$ and $y\in \mathscr K^{\perp\perp}$.
	Let $\lim_n(x_n+y_n)=z$, where  $x_n\in \mathscr H^{\perp\perp}$ and $y_n\in \mathscr K^{\perp\perp}$. It follows from \eqref{eq 02} that $\{x_n\}$ is a Cauchy sequence and so there is a $x\in \mathscr H^{\perp\perp}$ such that $\lim_nx_n=x$. Hence, there is a $y\in \mathscr K^{\perp\perp}$ such that $\lim_ny_n=y$. So $z=x+y$. This shows that $\mathscr H^{\perp\perp}+\mathscr K^{\perp\perp}$ is closed. Since $(\mathscr H^\perp,\mathscr K^\perp)$ is concordant, we conclude that $\mathscr H^{\perp\perp}+\mathscr K^{\perp\perp}$ is an orthogonally complemented submodule. \\
	Due to $(\mathscr H^\perp,\mathscr K^\perp)$ is concordant and $\alpha_0<1$ we have  $c_0((\mathscr H^{\perp\perp})_f, (\mathscr K^{\perp\perp})_f)<1$, which, by \cite[Theorem 12]{Deutsch}, yields that
	$$(\mathscr H^{\perp\perp}\cap \mathscr K^{\perp\perp})_f=(\mathscr H^{\perp\perp})_f\cap(\mathscr K^{\perp\perp})_f=\{0\}\,,\qquad(f\in \mathrm{S}(\mathscr A
	))\,.$$
	Hence, $\mathscr H^{\perp\perp}\cap \mathscr K^{\perp\perp}=\{0\}$. 
 \end{proof}
%%%%%%%%%%%%%%%%%%%%%%%%%%%%%%%%%%%%%%%%%%%%%%%%%%%%%%%%%%%%%%%%%%%%%%%%%%%%%%%%

\medskip
\section*{Disclosure statement}

On behalf of all authors, the corresponding author states that there is no conflict of interest. Data sharing is not applicable to this paper as no datasets were generated or analysed during the current study.

\section*{Funding}

The fourth author is supported by the National Natural Science Foundation of China (11971136), and the fifth author is supported by the Youth Backbone Teacher Training Program of Henan Province (2017GGJS140).

\end{document}